\documentclass[11pt]{amsart}
\usepackage{amsthm}
\usepackage{amsmath}
\usepackage[left=1.25in, right=1.25in]{geometry}
\usepackage{amssymb}
\usepackage{comment}
\usepackage{enumitem}
\usepackage{amsfonts}
\usepackage{mathtools}
\usepackage{hyperref}
\usepackage{todonotes}
\usepackage{lipsum}
\usepackage[utf8]{inputenc}

\newcommand{\C}{\mathbb{C}}

\newcommand\precdot{\mathrel{\ooalign{$\prec$\cr
  \hidewidth\raise0.001ex\hbox{$\cdot\mkern0.6mu$}\cr}}}

\newtheorem{theorem}{Theorem}[section]
\newtheorem{def-prop}[theorem]{Definition-Proposition}
\newtheorem{prop}[theorem]{Proposition}
\newtheorem{conj}[theorem]{Conjecture}
\newtheorem{lemma}[theorem]{Lemma}
\newtheorem{cor}[theorem]{Corollary}

\theoremstyle{definition}

\newtheorem{defin}[theorem]{Definition}

\theoremstyle{remark}
\newtheorem*{remark}{Remark}

\title{Spherical Schubert varieties and pattern avoidance}

\author{Christian Gaetz}
\thanks{The author was supported by a National Science Foundation Graduate Research Fellowship under Grant No. 1122374.}
\address{Department of Mathematics, Massachusetts Institute of Technology, Cambridge, MA 02139}
\email{\href{mailto:crgaetz@gmail.com}{{\tt crgaetz@gmail.com}}}

\date{\today}

\begin{document}
\begin{abstract}
A normal variety $X$ is called \emph{$H$-spherical} for the action of the complex reductive group $H$ if it contains a dense orbit of some Borel subgroup of $H$. We resolve a conjecture of Hodges--Yong by showing that their \emph{spherical permutations} are characterized by permutation pattern avoidance. Together with results of Gao--Hodges--Yong this implies that the sphericality of a Schubert variety $X_w$ with respect to the largest possible Levi subgroup is characterized by this same pattern avoidance condition.  
\end{abstract}

\maketitle

\section{Introduction}
\label{sec:intro}

\subsection{Spherical varieties}

Following \cite{brion-luna-vust, luna-vust}, a normal variety $X$ is called \emph{$H$-spherical} for the action of the complex reductive group $H$ if it contains a dense orbit of some Borel subgroup of $H$. Important examples of spherical varieties include projective and affine toric varieties, complexifications of symmetric spaces, and flag varieties (see Perrin's survey \cite{perrin}). Producing families of examples of and classifying spherical varieties is of significant interest \cite{Luna}. In this paper we resolve a conjecture of Hodges--Yong \cite{Spherical-paper}, thereby classifying (maximally) spherical Schubert varieties by a permutation pattern avoidance condition.

\subsection{Schubert varieties and pattern avoidance}

Let $G$ be a complex reductive algebraic group and $B$ be a Borel subgroup. The Bruhat decomposition decomposes $G$ as 
\[
G= \coprod_{w \in W} BwB,
\]
where $W$ denotes the Weyl group of $G$. The closures
\[
X_w=\overline{BwB/B}
\]
of the images of these strata in the flag variety $G/B$ are the \emph{Schubert varieties}, of fundamental importance in algebraic geometry and representation theory.

In the case $G=GL_n(\C)$, the Weyl group $W$ is the symmetric group $S_n$. Beginning with the groundbreaking result of Lakshmibai--Sandhya \cite{Lakshmibai} characterizing smooth Schubert varieties, it has been found that many important geometric and combinatorial properties (see, for example \cite{Karuppuchamy,woo-yong-gorenstein}) of $X_w$ are determined by \emph{permutation pattern avoidance} conditions on $w$. Let $w=w_1\ldots w_n \in S_n$ be a permutation written in one-line notation, and let $p \in S_k$ be another permutation. Then $w$ is said to have an \emph{occurrence} of the pattern $p$ at positions $1 \leq i_1 < \cdots < i_k \leq n$ if $w_{i_1}\ldots w_{i_k}$ are in the same relative order as $p_1 \ldots p_k$. If $w$ does not contain any occurrences of $p$, then $w$ is said to \emph{avoid} $p$.

Under the natural left action of $G$ on $G/B$, the stabilizer of $X_w$ is the parabolic subgroup $P_{J(w)} \subset G$ corresponding to the left descent set $J(w)$ of $w$. The parabolic subgroup $P_{J(w)}$ is not reductive, but contains the \emph{Levi subgroup} $L_{J(w)}$ as a maximal reductive subgroup. Following \cite{Spherical-paper}, we say $X_w$ is \emph{maximally spherical} if it is $L_{J(w)}$-spherical for the induced action of $L_{J(w)}$. Since all Schubert varieties are known to be normal by the work of DeConcini--Lakshmibai \cite{deconcini-lakshmibai} and Ramanan--Ramanathan \cite{rr}, this is equivalent to the existence of a dense orbit inside $X_w$ of a Borel subgroup of $L_{J(w)}$.

\subsection{Hodges and Yong's conjecture}

We consider the symmetric group $S_n$ as a Coxeter group with simple generating set $I=\{s_1,\ldots,s_{n-1}\}$, where $s_i$ is the adjacent transposition $(i \: i+1)$, and we write $J(w)$ for the left descent set of $w \in S_n$. See Section~\ref{sec:background} for background and basic definitions.

\begin{defin}[Hodges and Yong \cite{Spherical-paper}]
\label{def:spherical}
A permutation $w \in S_n$ is \emph{spherical} if it has a reduced word $s_{i_1}\cdots s_{i_{\ell(w)}}$ such that:
\begin{itemize}
    \item[(S.1)] $|\{t \: | \: s_{i_t}=s_j\}| \leq 1$ for $s_j \in I \setminus J(w)$, and
    \item[(S.2)] $|\{t \: | \: s_{i_t} \in C\}| \leq \ell(w_0(C)) + |C|$ for any connected component $C$ of the induced subgraph of the Dynkin diagram on $J(w)$.
\end{itemize}
\end{defin}

\begin{remark}
Hodges and Yong consider a more general class of spherical elements in finite Coxeter groups. Definition~\ref{def:spherical} is the special case which is relevant to Conjecture~\ref{conj:spherical-from-pattern-avoidance} ($J(w)$-spherical elements in $S_n$). 
\end{remark}

Spherical permutations were defined because of Conjecture~\ref{conj:spherical-variety-from-permutation}, which is proven in forthcoming work \cite{gao-hodges-yong} of Gao--Hodges--Yong.

\begin{conj}[Conjectured by Hodges and Yong \cite{Spherical-paper}; proof by Gao--Hodges--Yong \cite{gao-hodges-yong} in preparation]
\label{conj:spherical-variety-from-permutation}
The Schubert variety $X_w$ is maximally spherical if and only if $w$ is spherical.
\end{conj}

This geometric property is linked to permutation pattern avoidance by Conjecture~\ref{conj:spherical-from-pattern-avoidance}.

\begin{conj}[Hodges and Yong \cite{Spherical-paper}]
\label{conj:spherical-from-pattern-avoidance}
A permutation $w$ is spherical if and only if it avoids the twenty one patterns in $P$:
\begin{align*}
P&=\{24531,25314,25341,34512,34521,35412,35421,42531,45123,45213,45231,\\
&45312,52314,52341,53124,53142,53412,53421,54123,54213,54231\}.
\end{align*}
\end{conj}

Our main result resolves Conjecture~\ref{conj:spherical-from-pattern-avoidance}:

\begin{theorem}
\label{thm:main}
A permutation $w$ is spherical if and only if it avoids the patterns in $P$.
\end{theorem}

Combining this result with Gao--Hodges--Yong's proof of Conjecture~\ref{conj:spherical-variety-from-permutation}, we thus obtain a characterization of maximally spherical Schubert varieties in terms of pattern avoidance.

\begin{cor}
The Schubert variety $X_w$ is maximally spherical if and only if $w$ avoids the patterns from $P$.
\end{cor}

The following result, an immediate consequence of Theorem~\ref{thm:main}, was conjectured in \cite{Spherical-paper} and proven in \cite{Proper} using probabilistic methods.

\begin{cor}
\[
\lim_{n \to \infty} |\{\text{spherical permutations } w \in S_n\}|/n! = 0.
\]
\end{cor}
\begin{proof}
The Stanley--Wilf Conjecture, now a theorem of Marcus and Tardos \cite{Marcus-Tardos}, says that the number of permutations in $S_n$ avoiding any fixed set $Q$ of patterns is bounded above by $C^n$ for some constant $C$. Thus Theorem~\ref{thm:main} implies that 
\[
|\{\text{spherical permutations } w \in S_n\}|
\]
grows at most exponentially.
\end{proof}

\subsection{Outline}

Section~\ref{sec:background} recalls some basic definitions and facts about Bruhat order as well as a result of Tenner \cite{Tenner} characterizing Boolean intervals in Bruhat order. In Section~\ref{sec:divisible} we introduce the notion of \emph{divisible pairs} of permutations and connect these to Boolean permutations and spherical permutations. Divisible pairs, along with a helpful decomposition of the set $P$ of patterns, are applied in Section~\ref{sec:proof} to prove Theorem~\ref{thm:main}.

\section{Background} 
\label{sec:background}

\subsection{Bruhat order}
For $i=1,\ldots,n$, let $s_i$ denote the adjacent transposition $(i \: i+1)$ in the symmetric group $S_n$; the symmetric group is a Coxeter group with respect to the generating set $s_1,\ldots,s_{n-1}$ (see \cite{Bjorner-Brenti} for background on Coxeter groups). For $w \in S_n$, and expression
\[
w=s_{i_1}\cdots s_{i_{\ell}}
\]
of minimum length is a \emph{reduced word} for $w$, and in this case $\ell=\ell(w)$ is the \emph{length} of $w$.

The \emph{(right) weak order} is the partial order $\leq_R$ on $S_n$ with cover relations $w \lessdot_R ws_i$ whenever $\ell(ws_i)=\ell(w)+1$. The \emph{Bruhat order} is the partial order $\leq$ on $S_n$ with cover relations $w \lessdot wt$ for $t$ a 2-cycle such that $\ell(wt)=\ell(w)+1$. Both posets have the identity permutation $e$ as their unique minimal element.

For a permutation $w=w_1\ldots w_n \in S_n$ and integers $1 \leq a \leq b \leq n$, we write $w[a,b]$ for the set $\{w_a,w_{a+1},\ldots,w_b\}$. For two $k$-subsets $A,B$ of $[n]\coloneqq \{1,\ldots,n\}$ write $A \preceq B$ if $a_1 \leq b_1, \ldots, a_k \leq b_k$, where $A=\{a_1,\ldots,a_k\}$ and $B=\{b_1,\ldots,b_k\}$ with $a_1<\cdots<a_k$ and  $b_1<\cdots<b_k$. The following well-known property of Bruhat order will be useful:

\begin{prop}[Ehresmann \cite{Ehresmann}]
\label{prop:bruhat-implies-prefix-dominates}
Let $v,w \in S_n$, then $v \leq w$ if and only if 
\[
v[1,i] \preceq w[1,i]
\]
for all $i=1,\ldots,n$.
\end{prop}

A generator $s_i$ is a \emph{(left) descent} of $w$ if $\ell(s_iw)<\ell(w)$ (equivalently, if $w^{-1}(i+1)<w^{-1}(i)$). We write $J(w)$ for the set of descents of $w$. For any $J \subseteq \{s_1,\ldots,s_{n-1}\}$, we write $w_0(J)$ for the unique permutation of maximum length lying in the subgroup of $S_n$ generated by $J$. Explicitly, the one-line notation for $w_0(J)$ is an increasing sequence of consecutive decreasing runs, where decreasing run consists of $i+d,i+d-1,\cdots,i$ whenever $s_i,s_{i+1},\ldots, s_{i+d-1} \in J$ while $s_{i-1},s_{i+d} \not \in J$.

\subsection{Boolean permutations}

\begin{theorem}[Tenner \cite{Tenner}]
\label{thm:boolean-characterization}
The following are equivalent for a permutation $w \in S_n$:
\begin{itemize}
    \item[\normalfont{(1)}] The interval $[e,w]$ in Bruhat order is isomorphic to a Boolean lattice,
    \item[\normalfont{(2)}] No simple generator $s_i$ appears more than once in a reduced word for $w$,
    \item[\normalfont{(3)}] $w$ avoids the patterns $321$ and $3412$.
\end{itemize}
\end{theorem}

We will call a permutation satisfying the equivalent conditions of Theorem~\ref{thm:boolean-characterization} a \emph{Boolean} permutation. Theorem~\ref{thm:toric} suggests a connection between Boolean permutations and spherical varieties.

\begin{theorem}[Karuppuchamy \cite{Karuppuchamy}]
\label{thm:toric}
The Schubert variety $X_w$ is a toric variety if and only if $w$ is a Boolean permutation.
\end{theorem}

\section{Divisible pairs of permutations}
\label{sec:divisible}
\begin{defin}
Given a pair $(v,w)$ of permutations from $S_n$, we say that $(v,w)$ is \emph{divisible after position $i$} if 
\[
\left| v[1,i] \cap w[1,i] \right| \leq i-2, 
\]
and \emph{divisible at position $i$} if $v_i=w_i$ and 
\[
\left| v[1,i] \cap w[1,i] \right| \leq i-1.
\]
We say simply that $(v,w)$ is \emph{divisible} if there exists $1 \leq i \leq n$ such that $(v,w)$ is divisible at or after position $i$.
\end{defin}

\begin{prop}
\label{prop:divisible-equals-not-boolean}
A pair $(v,w)$ of permutations from $S_n$ is divisible if and only if $v^{-1}w$ is not Boolean.
\end{prop}
\begin{proof}
It is clear from the definition that $(v,w)$ is divisible if and only if $(uv,uw)$ is divisible for all $u \in S_n$, so it suffices to prove the case $v=e$.

Suppose that $w$ is not Boolean, so that $w$ contains a pattern $p \in \{321,3412\}$ by Theorem~\ref{thm:boolean-characterization}. If $p=3412$ occurs as $w_{i_1}w_{i_2}w_{i_3}w_{i_4}$ then $(e,w)$ is divisible after position $i_2$, since $w[1,i_2]=\{w_{i_1},w_{i_2}\}$ while $e[1,i_2]=\{1,\ldots,i_2\}$ must contain $w_{i_3}$ and $w_{i_4}$ if it contains either $w_{i_1}$ or $w_{i_2}$. If $p=321$ occurs as $w_{i_1}w_{i_2}w_{i_3}$, consider three cases: If $w_{i_2}=i_2$ then $(e,w)$ is divisible at $i_2$, since $w_{i_1}>i_2 \not \in e[1,i_2]$; If $w_{i_2}<i_2$, then $(e,w)$ is divisible after $i_2-1$, since $w_{i_2},w_{i_3}$ both lie in $e[1,i_2-1]$ but not in $w[1,i_2-1]$; Similarly, if $w_{i_2}>i_2$, then $(e,w)$ is divisible after $i_2$, since $w_{i_1},w_{i_2}$ both lie in $w[1,i_2]$ but not in $e[1,i_2]$.

Conversely suppose that $w$ is divisible. If $w$ is divisible after position $i$, then there are two elements $a<b \in w[1,i]$ which are not in $e[1,i]=\{1,\ldots,i\}$ and therefore also two elements $c,d \in w[i+1,n]$ with $c<d \in \{1,\ldots,i\}$. Either $w^{-1}(a)<w^{-1}(b)$ and $w^{-1}(c)<w^{-1}(d)$ in which case $w$ contains $3412$ or at least one of these statements fails and $w$ contains $321$; in either case $w$ is not Boolean. If $w$ is divisible at position $i$, then $w_i=e_i=i$ and there is some $a>i$ in $w[1,i-1]$ and some $b<i$ in $w[i+1,n]$; then the values $a,i,b$ form a $321$ pattern in $w$, so $w$ is not Boolean.
\end{proof}

\begin{prop}[Gao--Hodges--Yong \cite{gao-hodges-yong}]
\label{prop:alternate-definition-of-spherical}
A permutation $w \in S_n$ is spherical if and only if $w_0(J(w))w$ is a Boolean permutation.
\end{prop}

The following characterization of spherical permutations will be convenient for our arguments in Section~\ref{sec:proof}.

\begin{cor}
\label{cor:spherical-not-divisible}
A permutation $w \in S_n$ is spherical if and only if $(w_0(J(w)),w)$ is not divisible. 
\end{cor}
\begin{proof}
By Proposition~\ref{prop:alternate-definition-of-spherical}, $w$ is spherical if and only if $w_0(J(w))w$ is Boolean. Since $w_0(J(w))$ is an involution, Proposition~\ref{prop:divisible-equals-not-boolean} implies that this is equivalent to $(w_0(J(w)),w)$ not being divisible.
\end{proof}

\section{Proof of Theorem~\ref{thm:main}}
\label{sec:proof}
The following decomposition of the set $P$ of twenty one patterns will be crucial to the proof of Theorem~\ref{thm:main}: $P=P^{321} \cup P^{3412}$, where 
\begin{align*}
    P^{321}&=\{24531,25314,25341,42531,45231,45312,52314,52341,53124,53142,53412\} \\
    &=\{p \in P \: | \: w_0(J(p))p \text{ contains the pattern $321$}\},
\end{align*}
and
\begin{align*}
    P^{3412}&=\{34512,34521,35412,35421,45123,45213,45231,53412,53421,54123,54213,54231\} \\
     &=\{p \in P \: | \: w_0(J(p))p \text{ contains the pattern $3412$}\}.
\end{align*}
Notice that $45231$ and $53412$ lie in both $P^{321}$ and $P^{3412}$. A simple check shows that $P^{321}$ and $P^{3412}$ are also characterized by the following properties:
\begin{align}
    P^{321}=\{p \in S_5 \: | &\: p^{-1}(5)<p^{-1}(3)<p^{-1}(1), p^{-1}(4) \not \in [p^{-1}(5),p^{-1}(3)], \\ \nonumber &p^{-1}(2) \not \in [p^{-1}(3),p^{-1}(1)] \}, \\
    P^{3412}=\{p \in S_5 \: |& \: \max(p^{-1}(4),p^{-1}(5))<\min(p^{-1}(1),p^{-1}(2)),  \\ \nonumber &p^{-1}(3) \not \in [p^{-1}(4),p^{-1}(2)] \}.
\end{align}

The following proposition is obvious from the definitions, but will be useful to keep in mind throughout the proofs of Lemmas~\ref{lem:avoid-implies-spherical} and \ref{lem:spherical-implies-avoid}.

\begin{prop}
\label{prop:v-compared-to-w}
For $w \in S_n$, let $v=w_0(J(w))$ and $1 \leq a < b \leq n$. Then $v^{-1}(b)<v^{-1}(a)$ if and only if $w^{-1}(b)<w^{-1}(b-1)<\cdots < w^{-1}(a)$.
\end{prop}

\begin{lemma}
\label{lem:avoid-implies-spherical}
If $w \in S_n$ avoids the patterns from $P$ then $w$ is spherical.
\end{lemma}
\begin{proof}
We reformulate using Corollary~\ref{cor:spherical-not-divisible} and prove the contrapositive: if $(w_0(J(w)),w)$ is divisible, then $w$ contains a pattern from $P$.

\textit{Case 1:} Write $v$ for $w_0(J(w))$ and suppose that $(v,w)$ is divisible after $i$, and furthermore that $i$ is the smallest index for which this is true. Then we have:
\begin{align*}
v[1,i] \setminus w[1,i] &= \{a,b\}, \\
w[1,i] \setminus v[1,i] &= \{c,d\},
\end{align*}
where we may assume without loss of generality that $a<b$ and $c<d$. We have $v \leq_R w$, so in particular $v \leq w$ in Bruhat order; thus by Proposition~\ref{prop:bruhat-implies-prefix-dominates} we must have $a<c$ and $b<d$. Suppose that $c \leq b+1$; if $c<b$, then, since $c$ appears after $b$ in $v$, it must be that $b,c$ lie in the same decreasing run of $v$, but by Proposition~\ref{prop:v-compared-to-w} this implies that $b$ appears before $c$ in $w$, a contradiction. If $c=b+1$, then $s_c$ is a descent of $w$, so $c$ appears before $b$ in $v$, again a contradiction. Thus we have $a<b<b+1<c<d$. 

We wish to conclude that $w$ contains a pattern from $P$. If any value $x \in \{b+1,b+2,\ldots,c-1\}$ does not lie between $b$ and $c$ in $w$, then we are done, since the values $\{a,b,c,d,x\}$ form a pattern from $P^{3412}$ in $w$. Otherwise, all of these values appear between $b$ and $c$ in $w$. Suppose that they do not appear in decreasing order, so $w^{-1}(b+j)<w^{-1}(b+j+1)$ for some $j+1<c-b$. Then the values $\{c,d,b+j,b+j+1,a,b\}$ either contain a pattern from $P^{3412}$, or appear in $w$ in the order $c,b+j,d,a,b+j+1,b$. In this last case $c,b+j,a,b+j+1,b$ forms an occurrence of the pattern $53142$ from $P^{321}$. Finally, suppose that $\{b+1,b+2,\ldots,c-1\}$ appear in decreasing order in $w$ between $b$ and $c$; then by Proposition~\ref{prop:v-compared-to-w} $c$ appears before $b$ in $v$, a contradiction. Thus in all cases $w$ contains a pattern from $P$.

\textit{Case 2:} Write $v$ for $w_0(J(w))$ and suppose that $(v,w)$ is divisible at $i$, and furthermore that $i$ is the smallest index at or after which $(v,w)$ is divisible. Then we have $v_i=w_i$ and \begin{align*}
    v[1,i-1] \setminus w[1,i-1] &= \{a\}, \\
w[1,i-1] \setminus v[1,i-1] &= \{c\},
\end{align*}
with $a<c$ by Proposition~\ref{prop:bruhat-implies-prefix-dominates}. We claim that $a$ is the minimal element in a decreasing run of $v$. Indeed, otherwise $a-1$ appears immediately after $a$ in $v$, and thus $a-1$ also appears after $a$ in $w$. But then $a-1 \in v[1,i] \setminus w[1,i]$, contradicting the minimality of $i$, thus $a$ is the minimal element in a decreasing run, and is smaller than all values appearing after it in $v$. Similarly, $c$ is the maximal element in a decreasing run of $v$ and is larger than all values appearing before it in $v$. Also note that $a<v_i-1$ and $c>v_i+1$, for if $a=v_i-1$ then $w^{-1}(a+1) < w^{-1}(a)$ but $v^{-1}(a+1)>v^{-1}(a)$, contradicting Proposition~\ref{prop:v-compared-to-w}, and similarly for $c$.

We will now see that $c,v_i,a$ participate in an occurrence in $w$ of some pattern $p \in P$. Suppose first that all values $c-1,c-2,\ldots,v_i+1$ lie in between $c$ and $v_i$ in $w$. If these occur in decreasing order, then $c$ and $v_i$ must occur in the same decreasing run of $v$, but this is not the case since $c$ appears at the beginning of its run, but after $v_i$ in $v$. Thus there is some $j \leq c-v_i-2$ such that $w^{-1}(c-j-1)<w^{-1}(c-j)$. In this case $c,c-j-1,c-j,v_i,a$ form an occurrence in $w$ of the pattern $53421 \in P^{3412}$. Similarly, if all values $v_i-1,v_i-2,\ldots,a+1$ lie in between $v_i$ and $a$ in $w$, then $w$ contains an occurrence in $w$ of the pattern $54231 \in P$.

In the only remaining case, there is some $x \in \{c-1,c-2,\ldots,v_i+1\}$ not lying between $c$ and $v_i$ in $w$ and some $y \in \{v_i-1,v_i-2,\ldots,a+1\}$ not lying between $v_i$ and $a$ in $w$. Then the values $\{c,v_i,a,x,y\}$ form a pattern $p$ from $P^{321}$ in $w$, with $c,v_i,a$ corresponding to the values $5,3,1$ in $p$ respectively and $x,y$ corresponding to $4,2$.
\end{proof}

\begin{lemma}
\label{lem:spherical-implies-avoid}
If $w\in S_n$ is spherical then $w$ avoids the patterns from $P$.
\end{lemma}
\begin{proof}
We will apply Proposition~\ref{prop:alternate-definition-of-spherical} and prove the contrapositive: if $w$ contains a pattern $p$ from $P$, then $vw$ is not Boolean, where $v=w_0(J(w))$.

Suppose first that $w$ contains a pattern $p$ from $P^{321}$ and that $w_i,w_j,w_k$ with $i<j<k$ correspond to the values $5,3,1$ in $p$ respectively. Since the $2$ and $4$ in the pattern $p$ do not lie between $w_j,w_k$ and $w_i,w_j$ respectively, Proposition~\ref{prop:v-compared-to-w} implies that $v(w_i)>v(w_j)>v(w_k)$ since $v=v^{-1}$. Thus $vw$ contains the pattern $321$ and is not Boolean by Theorem~\ref{thm:boolean-characterization}.

Suppose now that $w$ contains a pattern $p$ from $P^{3412}$. Let $w_i,w_j,w_k,w_{\ell}$ with $i<j<k<\ell$ correspond to the values $\{1,2,4,5\}$ from $p$ (thus one of $w_i,w_j$ corresponds to $4$ and the other to $5$, while one of $w_k,w_{\ell}$ corresponds to $1$ and the other $2$). Since the $3$ in the pattern $p$ does not lie between the $2$ and the $4$, Proposition~\ref{prop:v-compared-to-w} implies that $\min(v(w_i),v(w_j))>\max(v(w_k),v(w_{\ell}))$. Thus either $vw$ contains the pattern $3412$ in these positions or contains a $321$ pattern in some subset of them. In either case $vw$ is not Boolean.
\end{proof}

Lemmas~\ref{lem:avoid-implies-spherical} and \ref{lem:spherical-implies-avoid} together yield Theorem~\ref{thm:main}.

\bibliographystyle{plain}
\bibliography{main}
\end{document}